\documentclass[11pt]{article}
\usepackage[letterpaper,left=1in,right=1in,top=1in,bottom=1in]{geometry}
\usepackage[mathscr]{eucal}
\usepackage{epsfig,epsf,psfrag}
\usepackage{amssymb,amsfonts,latexsym}
\usepackage{amsmath}
\usepackage{graphicx}
\usepackage{epstopdf}
\usepackage{bm,xcolor,url}
\usepackage{fixltx2e}
\usepackage{array}
\usepackage{verbatim}
\usepackage{algpseudocode, cite}
\usepackage{algorithm}
\usepackage{verbatim}
\usepackage{textcomp}
\usepackage{mathrsfs}
\usepackage[referable]{threeparttablex}
\usepackage{subfig}
\usepackage{amsthm}
\usepackage{enumerate}
\usepackage{footnote}
\usepackage{enumitem}
\usepackage{xr}
\usepackage{mathtools}
\catcode`~=11 \def\UrlSpecials{\do\~{\kern -.15em\lower .7ex\hbox{~}\kern .04em}} \catcode`~=13 

\allowdisplaybreaks[3]

\newcommand{\tnorm}[1]{{\left\vert\kern-0.25ex\left\vert\kern-0.25ex\left\vert #1 
    \right\vert\kern-0.25ex\right\vert\kern-0.25ex\right\vert}}
\newcommand{\tnormt}[1]{{\vert\kern-0.25ex\vert\kern-0.25ex\vert #1 
    \vert\kern-0.25ex\vert\kern-0.25ex\vert}}

\newcommand{\abst}[1]{\vert#1\vert}

\newcommand{\nn}{\nonumber}

\newcommand{\numberthis}{\addtocounter{equation}{1}\tag{\theequation}}

\newcommand{\dom}{\mathsf{dom}\,}

\newcommand{\ri}{\mathsf{ri}\,}

\newcommand{\calF}{\mathcal{F}}

\newcommand{\calI}{\mathcal{I}}

\newcommand{\calK}{\mathcal{K}}

\newcommand{\calR}{\mathcal{R}}

\newcommand{\calW}{\mathcal{W}}
\newcommand{\calX}{\mathcal{X}}

\newcommand{\rmA}{\mathrm{A}}

\newcommand{\rmD}{\mathrm{D}}

\newcommand{\bbE}{\mathbb{E}}

\newcommand{\bbR}{\mathbb{R}}
\newcommand{\bbS}{\mathbb{S}}

\newcommand{\bbV}{\mathbb{V}}

\DeclareMathAlphabet{\mathbsf}{OT1}{cmss}{bx}{n}

\newcommand{\rvF}{\mathsf{F}}

\newcommand{\rvT}{\mathsf{T}}

\newcommand{\tilH}{\widetilde{H}}

\newcommand{\tilM}{\widetilde{M}}

\newcommand{\tilw}{\widetilde{w}}

\newcommand{\barw}{\bar{w}}

\newcommand{\ipt}[2]{\langle{#1},{#2}\rangle}

\DeclareMathOperator*{\argmax}{arg\,max}

\DeclareMathOperator{\st}{s.t.}

\DeclareMathOperator{\tr}{tr}

\newtheorem{theorem}{Theorem} 
\newtheorem*{theorem*}{Theorem}
\newtheorem{lemma}{Lemma}

\newtheorem{prop}{Proposition}
\newtheorem{corollary}{Corollary}
 
\newtheorem*{assump*}{Assumption}

\theoremstyle{definition}
 
\newtheorem{example}{Example} 

\theoremstyle{remark}
\newtheorem{remark}{Remark}

\newcommand{\qednew}{\nobreak \ifvmode \relax \else
      \ifdim\lastskip<1.5em \hskip-\lastskip
      \hskip1.5em plus0em minus0.5em \fi \nobreak
      \vrule height0.75em width0.5em depth0.25em\fi}

\title{
An Optimization Perspective on the Monotonicity of the Multiplicative Algorithm for Optimal Experimental Design}
\usepackage[colorlinks=true,breaklinks=true,bookmarks=true,urlcolor=blue,
     citecolor=blue,linkcolor=blue,bookmarksopen=false,draft=false]{hyperref}
\author{Renbo Zhao
}
\usepackage{xfrac}

\begin{document}

\maketitle

\begin{abstract}
We provide an optimization-based argument for the monotonicity of the multiplicative algorithm (MA) for a class of  optimal experimental design problems considered in Yu~\cite{Yu_10}. Our proof avoids introducing auxiliary variables (or problems) and leveraging statistical arguments,   and is much more straightforward and simpler compared to the proof in~\cite[Section~3]{Yu_10}. The simplicity of our monotonicity proof also allows us to easily identify several sufficient conditions  that ensure the {strict monotonicity} of MA. In addition, we provide two simple and similar-looking examples on which MA behaves very differently. 
These examples offer insight in the behaviors of MA, and also reveal some limitations of MA when applied to certain optimality criteria. We discuss these limitations, and  pose open problems that may lead to deeper understanding of the behaviors of MA  on these optimality criteria. 
\end{abstract}

\section{Introduction} \label{sec:intro}

Optimal experimental design (OED) is an interesting and important field that lies at the intersection of statistics and optimization, and has a long history of development (see e.g., Fedorov~\cite{Fedorov_72}, Silvey~\cite{Silvey_80}, Pukelsheim~\cite{Pukel_06}). Depending on the purpose of the experimenter,  there are many possible formulations of the OED problem, which can lead to either continuous or discrete optimization problems. In this work, we are interested in the following (finite-dimensional) continuous formulation of the OED problem. Suppose that we are interested in estimating some (deterministic) 
parameter $\theta\in\bbR^d$ through a sequence of experiments. In each experiment, given a design point $x\in\bbR^q$,  the (conditional) probability density function (PDF) of the 
response $Y$ is modeled as $p_{Y|X}(y|x;\theta)$, 
which is  parameterized by $\theta$ and is assumed to be known. 
   We focus on a finite design space $\calX := \{x_1,\ldots,x_n\}$, and seek a design measure $w\in \Delta_n:=\{w\ge 0: \sum_{i=1}^n w_i = 1\}$, which is a distribution on $\calX$, such that certain real-valued function (known as the optimality criterion) of its {\em moment matrix} $M_\theta(w)$ is maximized, where 
\begin{equation}
M_\theta(w):= \bbE_{X\sim w}[I_{Y|X}(\theta)] := \textstyle\sum_{i=1}^n w_i\; I_{Y|X=x_i}(\theta),
\end{equation}
and $I_{Y|X=x_i}(\theta)$ denotes the Fisher information matrix about $\theta$ with respect to the conditional 
PDF $p_{Y|X}(y|x_i;\theta)$, namely
\begin{equation}
I_{Y|X=x_i}(\theta):= \bbE\left[s_{Y|X=x_i}(\theta) s_{Y|X=x_i}(\theta)^\top \right], \quad s_{Y|X=x_i}(\theta) := \frac{\partial}{\partial \theta} \ln p_{Y|X}(y|x_i;\theta),\quad \forall\,i\in[n].  
\end{equation}
(Here $[n]:=\{1,\ldots,n\}$.)
Intuitively, $M_\theta(w)$ measures the the amount of information about $\theta$ contained in the response $Y$, averaged over the distribution of (random) design point $X$. 
For notational convenience,  define $A_i(\theta):= I_{Y|X=x_i}(\theta)$ for $i\in[n]$. 

Note that in general,  $A_i(\theta)$ may depend on the (unknown) parameter $\theta$ (for $i\in[n]$),  which poses certain difficulties in formulating the OED problem. 
To resolve this issue,  one common approach  in the literature  
is to substitute an a priori estimate of $\theta$, denoted by $\theta_0$, into the definition of $A_i(\theta)$, and  $\theta_0$ can be obtained directly from domain knowledge or estimated from a pilot sample. This results in the so-called ``locally optimal design'' (see e.g.,~\cite{Chernoff_53,Yang_12}). 
Of course, 
such an approach ignores the uncertainty of $\theta$, but it also has clear advantages --- 
it leads to a relatively simple OED formulation, and also works well when the dependence of $A_i(\theta)$ on $\theta$ is ``weak''. In this work, we shall adopt this approach, and hence suppress the dependence of $A_i(\theta)$ and $M_\theta(w)$ on $\theta$. As a result, we introduce simpler notations, namely $A_i:= A_i(\theta)$ for $i\in[n]$ and $M(w):=M_\theta(w)$. 

Let us now introduce the optimization problem associated with OED. We first introduce some standard notations. Let $\bbS^d$, $\bbS_{+}^d$ and $\bbS_{++}^d$ denote the sets of $d\times d$ symmetric, symmetric and positive semi-definite, and symmetric and positive definite matrices, respectively. For $A,B\in\bbS^d$, we write $A\succeq B$ if $A - B \in \bbS_{+}^d$ and $A\succ B$ if $A - B \in \bbS_{++}^d$. From the definition of $A_i$ above, it is clear that $A_i\succeq 0$ for all $i\in[n]$. In addition, we shall assume that $A_i\ne 0$ for $i\in[n]$ and $\sum_{i=1}^n A_i \succ 0$. 
Given an optimality criterion  $\phi:\bbS_{++}^d \to \bbR$, 
the optimization problem reads: 
\begin{equation}
\textstyle \sup_{w\in\Delta_n }\; \phi(M(w)), \quad {\rm where}\quad  M(w):= \sum_{i=1}^n w_i A_i. \tag{${\rm OED}_0$} \label{eq:OED0} 
\end{equation}
%
In the literature, $\phi$ 
is typically assumed to be concave and isotonic on $\bbS_{++}^d$ (cf.~\cite[Chapter~5]{Pukel_06}), and hence~\eqref{eq:OED} is a convex optimization problem. 
Note that we call $\phi$  isotonic on $\bbS_{++}^d$ if 
\begin{equation}
0\prec A \preceq B \quad\Longrightarrow\quad  \phi(A) \le \phi(B).  \label{eq:isotone}
\end{equation} 
For the purpose of this work, we shall also assume $\phi$ to be differentiable on $\bbS_{++}^d$. 
Typical examples of $\phi$ include
\begin{itemize}
\item the D-criterion: $\phi_\rmD(M) := \ln\det(M)$ for $M\succ 0$,
\item the A-criterion: $\phi_\rmA(M) := -\tr(M^{-1})$ for $M\succ 0$, and more generally,
\item the $p^{\rm th}$-mean-criterion: $\phi_p(M) := -\tr(M^{-p})$ for  $p >0$ and $M\succ 0$.
\end{itemize}
Now, note that for some $w\in\Delta_n$, we may have $M(w)\not \in \bbS_{++}^d$, 
and the value of $\phi$ is undefined at $M(w)$.  Therefore, to make~\eqref{eq:OED0} well-posed, we extend the definition of $\phi$ to $\bbS^n$ by %
defining a new function 
$\Phi:\bbS^n\to \underline \bbR:=\bbR\cup\{-\infty\}$ such that 
\begin{equation}
   \Phi(M)  := \bigg\{ 
   \begin{aligned}
   \phi(M), &\quad   M\succ 0  \\
   -\infty, &\quad    \mbox{otherwise}
\end{aligned}\;\;. 
\end{equation}
(Note that $\Phi$ is a function of $\phi$.)  We then solve the following problem:
\begin{equation}
f^*:= \textstyle \sup_{w\in\Delta_n }\; \{f(w):=\Phi(M(w))\}.  \tag{${\rm OED}$} \label{eq:OED} 
\end{equation}
The formulation in~\eqref{eq:OED} restricts the feasible moment matrix $M(w)$ to be positive definite, and due to this, 
the feasible region of~\eqref{eq:OED} is given by 
\begin{equation}
\Delta_n^+:=\{w\in \Delta_n: M(w)\succ 0\}. 
\end{equation}
(In Section~\ref{sec:discuss}, we will introduce a ``generalized'' formulation of~\eqref{eq:OED}. However, 
in this work, we shall stick to~\eqref{eq:OED}  since it is more convenient for 
us to 
develop the theory of the multiplicative algorithm, which will be introduced shortly.)  

Developing numerical algorithms to solve
~\eqref{eq:OED} has attracted much research 
efforts in the past fifty years, from both the statistics and the optimization community. As a result, several effective algorithms  have been developed --- for a non-exhaustive list of works, see~\cite{Fedorov_72,Wynn_72,Atwood_73,Fellman_74,Titter_76,Silvey_78,Torsney_83,Bohn_86,Kachiyan_96,Pazman_06,Todd_07,
Dette_08,Ahi_08,Harman_09,Torsney_09,Yu_10,Yu_11,Lu_13,Ahi_15,Lu_18,Cohen_19,Zhao_22mg,Zhao_23,Zhao_25}. 
Note that the majority of 
these works 
solely tackle the D-optimal design problem, namely $\phi:= \phi_\rmD$ in~\eqref{eq:OED}, while other works consider a more general setting, where $\phi$ belongs to  a class of optimality criteria in~\eqref{eq:OED} (see e.g.,~\cite{Silvey_78,Yu_10,Lu_13}). 
In fact, among all of the algorithms proposed, the {\em multiplicative algorithm} (MA), first introduced in~\cite{Silvey_78}, is one of the most widely adopted algorithms for solving~\eqref{eq:OED}, and have received extensive research efforts (see e.g.,~\cite{Fellman_74,Titter_76,Silvey_78,Torsney_83,Pazman_06,
Dette_08,Harman_09,Torsney_09,Yu_10,Cohen_19,Zhao_22mg}). This algorithm has an extremely simple form, which is presented in Algorithm~\ref{algo:MA}. (In the following, we shall use MA and Algorithm~\ref{algo:MA} interchangeably.)  
The popularity of 
MA is due to at least three reasons. First,  it incurs  low computations per iteration. In fact, the only non-trivial computation involves computing the gradient $\nabla f(w^k)$, which stands in contrast to the Newton-type methods (e.g.,~\cite{Lu_13}) that involve computing and manipulating the Hessian of $f$. Second, the  implementation of MA is extremely simple, and involves minimal choices of parameters. Indeed, one only needs to choose the power parameter $\lambda\in(0,1]$ 
before the algorithm starts, and no parameters needs to be computed subsequently. 
This 
is clearly an advantage over the Frank-Wolfe-type methods~
\cite{Fedorov_72,Wynn_72,Atwood_73,Bohn_86,Kachiyan_96,Todd_07,Ahi_08,Zhao_23,Zhao_25}, where each iteration involves judicious computation of the step-size (either in closed-form or via line-search). Third,  MA has wide applicability and theoretical soundness. In fact, as shown in Yu~\cite{Yu_10}, 
the sequence of objective values $\{f(w^k)\}_{k\ge 0}$ generated by MA 
monotonically 
converges to $f^*$ under a variety of optimality criteria (which include  
all the criteria mentioned above). 
Moreover, when $\phi= \phi_\rmD$, our recent work~\cite{Zhao_22mg} showed that MA enjoys an ergodic $O(1/k)$ convergence rate in terms of the objective value, i.e., $f^* - f(\barw^k) = O(1/k)$, where $\barw^k: = (1/k) \sum_{i=0}^{k-1} w^i$ for $k\ge 1$ (see~\cite{Cohen_19} for the  ergodic $O(1/k)$ convergence rate in  terms of another criterion $\max_{i=1}^n \ln(\nabla_i f(\barw^k))$). 
Before concluding our brief review on the literature, it is worth mentioning that~\eqref{eq:OED} is a challenging problem from the viewpoint of first-order methods. Indeed,  for almost all the optimality criterion $\phi$ that we are interested in (which include all the criteria mentioned above), the objective function $f$ does not have Lipschitz (or H\"olderian) function value or gradient on the feasible region $\Delta_n^+$.

\begin{algorithm}[t!]
\caption{Multiplicative Algorithm for Solving~\eqref{eq:OED}}\label{algo:MA}
\begin{algorithmic}
\State {\bf Input}: Power parameter $\lambda\in(0,1]$ and starting point $w^0 \in \ri \Delta_n:= \{w> 0: \sum_{i=1}^n w_i = 1\}$ 
\State {\bf At iteration $k\ge 0$}:
\begin{enumerate}[leftmargin = 6ex]
\item Compute $\nabla f(w^k)$, namely the gradient of $f$ at $w^k$. 
\item \label{step:multiply} Compute $\barw^k := w^k\circ \nabla f(w^k)^\lambda$, where $\circ$ denotes the entrywise product and $(\cdot)^\lambda$ is applied entrywise to $\nabla f(w^k)$.
\item $w^{k+1}:= \barw^k/\sum_{i=1}^n \barw_i^k$. 
\end{enumerate}
\end{algorithmic}
\end{algorithm}

\vspace{1ex}

\noindent
In this work, we shall focus on the monotonicity of the sequence of objective values $\{f(w^k)\}_{k\ge 0}$ generated by MA. Such a monotonicity property is desirable  arguably for any optimization algorithm, and often plays an important role in analyzing the asymptotic convergence and/or convergence rate of the 
algorithm. In fact, this property is the ``driving force'' in proving the asymptotic convergence of MA for solving~\eqref{eq:OED} in~\cite{Yu_10}. 
While monotonicity is easy to see for the ``classical'' gradient-type algorithms (under the Lipschitz-gradient condition on $f$ and proper choice of step-sizes), it is much harder to establish for MA on~\eqref{eq:OED}. That said, in the seminal work~\cite{Yu_10}, Yu showed the monotonicity of MA for  a class of optimality criteria $\phi$. 
To describe this class of optimality criteria, Yu defined a new function $\psi:\bbS_{++}^d \to \bbR$ based on $\phi$, namely 
\begin{equation}
\psi(M) := -\phi(M^{-1}), \quad \forall\,M\succ 0, \label{eq:def_psi}
\end{equation}
and he placed the following assumptions on $\psi$: 
\begin{enumerate}[label = (A\arabic*)]
\item \label{assump:diff} $\psi$ is differentiable on $\bbS_{++}^d$.
\item \label{assump:iso} $\psi$ is isotonic on $\bbS_{++}^d$ (cf.~\eqref{eq:isotone}), which amounts to $\nabla \psi(M)\succeq 0$ for all $M\in \bbS_{++}^d$ under~\ref{assump:diff}. 
\item \label{assump:concave} $\psi$ is concave on $\bbS_{++}^d$: for any $X,Y\in \bbS_{++}^d$, we have $\psi(Y) \le \psi(X) + \ipt{\nabla \psi(X)}{Y-X}$ . 
\end{enumerate}
Note that these assumptions hold for the D-, A- and the $p^{\rm th}$-mean-criteria with $p\in(0,1)$. In proving the monotonicity of MA,  Yu made use of statistical arguments that  were 
inspired from the EM algorithm~\cite{Dempster_77}. Specifically, he introduced several  
auxiliary problems that are defined on the augmented variable space, and reduce to the original problem~\eqref{eq:OED} upon partial minimization. He then showed that MA can be regarded as an algorithm that improves the objective value of one of the auxiliary problems, and hence improves the objective value of~\eqref{eq:OED}. Although these arguments bear certain statistical intuitions, they do require introducing several auxiliary variables and transferring between different auxiliary problems, and hence are somewhat convoluted. In addition, the arguments leveraged some results in statistical estimation theory that the optimization audience may not be familiar with.  

The main contribution of this work is to provide an optimization-based argument for the monotonicity of MA under the same assumptions made in Yu~\cite{Yu_10}, namely~\ref{assump:diff} to~\ref{assump:concave}. As we shall see,  our proof does not need to introduce any auxiliary variables or problems, or leverage any statistical arguments.  In fact, it is much more straightforward and simpler compared to the proof in~\cite[Section~3]{Yu_10}. The crux of our proof is to make use of the 
{\em matrix Cauchy-Schwartz} (MCS) inequality~\cite{Lav_08}, which is simple to prove but less-known. Indeed,  the finite-sum structure of $M(w)$, together with the functional form of $\psi$ in~\eqref{eq:def_psi}, makes the MSC inequality particularly suitable 
for analyzing MA on~\eqref{eq:OED}. The simplicity of our monotonicity proof also allows us to easily identify several conditions on $\psi$ and $\lambda$ that ensure the {\em strict monotonicity} of MA, which plays an important role in the convergence analysis of MA on~\eqref{eq:OED} (cf.~\cite[Theorem~3]{Yu_10}). In addition, we provide two simple and similar-looking examples on which MA behaves very differently. 
These examples not only demonstrate the advantages of choosing $\lambda\in(0,1)$ as opposed to $\lambda=1$ in terms of ensuring the convergence of MA, but also reveal some limitations of MA when applied to~\eqref{eq:OED}  with certain optimality criteria $\phi$ (e.g., the c-criterion). 
We conclude this paper by discussing these limitations,
and  pose open problems that may lead to deeper understanding of the behaviors of MA  on these optimality criteria. 

\section{Proof of the Monotonicity of MA}

Before proving the monotonicity of MA, let us first digress a bit and examine 
 the transformation 
$\rvT:\phi\mapsto \psi$, where $\psi$ is given in~\eqref{eq:def_psi}. 
As mentioned in Section~\ref{sec:intro}, this transformation plays  an important role in~\cite{Yu_10}  for identifying the class of optimality criteria $\phi$ on which MA is monotonic. 
Indeed, as we shall see below,  
$\rvT$ has some nice properties that may be of independent interest. 
 To that end, 
 let $\bbV$ be  the vector space consisting of all the functions $\phi: \bbS_{++}^d\to \bbR$ that are {\em differentiable} on $\bbS_{++}^d$, and define the convex cone 
\begin{equation}
\calK:= \{\phi\in\bbV:\; \phi \;\,\mbox{is isotonic on}\;\, \bbS_{++}^d\}. 
\end{equation}

\begin{lemma} \label{lem:T}
For any $\phi\in\bbV$, define the function $\rvT (\phi):\bbS_{++}^d\to \bbR$ such that  
\begin{equation}
(\rvT (\phi))(M) := -\phi(M^{-1}), \quad \forall\,M\succ 0. \label{eq:def_T}
\end{equation}
The transformation $\rvT$ 
is a linear automorphism on $\bbV$ with $\rvT^{-1} = \rvT$. In addition, its restriction on $\calK$, denoted by by $\rvT_\calK$, is an automorphism on $\calK$ with $\rvT_\calK^{-1} = \rvT_\calK$. 
\end{lemma}

\begin{proof}
From~\eqref{eq:def_T}, it is clear that $\rvT$ is a linear operator on $\bbV$. 
For any $\phi\in\bbV$, define $\psi:= \rvT (\phi)$. From~\eqref{eq:def_T}, it is clear that $\psi$ is differentiable on $\bbS_{++}^d$ and hence $\psi\in\bbV$. As a result, $\rvT(\bbV)\subseteq \bbV$. Now, if $\rvT(\phi_1)=\rvT(\phi_2)$ for some $\phi_1, \phi_2\in\bbV$, then $\phi_1(M^{-1}) = \phi_2(M^{-1})$ for all $M\succ 0$, which amounts to $\phi_1(M) = \phi_2(M)$ for all $M\succ 0$, and hence $\phi_1 = \phi_2$. This shows that $\rvT$ is one-to-one. Also, since 
\begin{equation}
\rvT(\rvT(\phi)) = \phi, \quad \forall\, \phi\in\bbV, \label{eq:T2}
\end{equation}
we know that 
$\rvT^{-1} = \rvT$ and $\rvT$ is onto. This shows that $\rvT$ is a linear automorphism on $\bbV$. 
Now, denote restriction of $\rvT$ on $\calK$ by $\rvT_\calK$.  For any $\phi\in\calK$, since the mapping $M\to M^{-1}$ is antitonic on $\bbS_{++}^d$ (namely if $0\prec A \preceq B$, 
then $0\prec B^{-1} \preceq A^{-1}$), we know that $\psi\in\calK$, 
and hence $\rvT_\calK(\calK)\subseteq \calK$. Since $\rvT$ is one-to-one on $\bbV$, it is clear that  $\rvT_\calK$ is one-to-one on $\calK$. Finally, by~\eqref{eq:T2} and $\rvT_\calK(\calK)\subseteq \calK$, we know that $\rvT_\calK^{-1} = \rvT_\calK$ and $\rvT_\calK$ is onto. 
\end{proof}

\begin{remark} \label{rmk:concave}
Given $\phi\in\calK$ that is concave on $\bbS_{++}^d$, note that  $\psi:=\rvT(\phi)$ may not be convex or concave on $\bbS_{++}^d$. For a simple example, consider $d=1$ and $\phi(t):= -e^{-t}$ for $t>0$. As a result, $\psi(t) = e^{-1/t}$ for $t>0$, which is convex on $(0,1/2]$ and concave on $[1/2, +\infty)$. 
\end{remark}

\noindent
Note that Lemma~\ref{lem:T} will not directly appear in our proof of the monotonicity of MA (cf.~Theorem~\ref{thm:monotone}), but it facilitates our exposition below. 
Next, we present a simple formula that expresses 
$\nabla\phi$ in terms of  
$\nabla \psi$, 
where $\psi:=\rvT(\phi)$. 
The proof of this formula is standard, and deferred to Appendix~\ref{app:diff}.  

\begin{lemma} \label{lem:differential}
Given $\phi\in \bbV$,  let $\psi:=\rvT(\phi)\in\bbV$. 
For any $M\succ 0$, we have $\nabla \phi(M) = M^{-1} \nabla\psi(M^{-1}) M^{-1}$,  and hence $$\ipt{\nabla \phi(M)}{M} = \ipt{ \nabla\psi(M^{-1}) }{M^{-1}}.$$
\end{lemma}

%

Next, let us introduce the {\em matrix Cauchy-Schwartz} (MCS) inequality (see e.g.,~\cite{Trip_99,Lav_08}). 
In the next lemma, we present a version of the MCS inequality that is slightly different from 
the literature. For readers' convenience, we include its proof in Appendix~\ref{app:MCS}.  


\begin{lemma}[MCS inequality] \label{lem:MCS}
Let $A_i\in \bbR^{q\times p}$ and $B_i\in \bbR^{d\times p}$ for $i\in[n]$, such that $\sum_{i=1}^n A_iA_i^\top \succ 0$. Then we have 
\begin{equation}
\textstyle \sum_{i=1}^n B_iB_i^\top \succeq  (\sum_{i=1}^n B_iA_i^\top) (\sum_{i=1}^n A_iA_i^\top)^{-1}(\sum_{i=1}^n A_iB_i^\top), \label{eq:CS_2}
\end{equation}
and the equality holds if and only if $B_i = (\sum_{i=1}^n B_iA_i^\top) (\sum_{i=1}^n A_iA_i^\top)^{-1}A_i$ for all $i\in[n]$.
\end{lemma}

\noindent
From Lemma~\ref{lem:MCS}, we can easily obtain the following corollary. 

\begin{corollary}\label{cor:CS}
Let $V_i\succeq 0$ for $i\in[n]$ and $\alpha_i,\beta_i\ge 0$ for $i\in[n]$,  such that $\sum_{i=1}^n \alpha_i V_i\succ 0$. Then
\begin{equation}
\textstyle \sum_{i=1}^n \beta_i V_i\; \succeq\; \left(\sum_{i=1}^n \sqrt{\alpha_i\beta_i} V_i\right) \left(\sum_{i=1}^n \alpha_i V_i\right)^{-1} \left(\sum_{i=1}^n \sqrt{\alpha_i\beta_i} V_i\right). \label{eq:CS}
\end{equation}
The equality holds if and only if $\sqrt{\beta_i} V_i^{1/2} = \sqrt{\alpha_i}\left(\sum_{i=1}^n \sqrt{\alpha_i\beta_i} V_i\right) \left(\sum_{i=1}^n \alpha_i V_i\right)^{-1} V_i^{1/2}$ for all $i\in[n]$.
\end{corollary}

\begin{proof}
For $i\in[n]$, set $A_i = \sqrt{\alpha_i} V_i^{1/2}$ and $B_i = \sqrt{\beta_i} V_i^{1/2}$ in Lemma~\ref{lem:MCS}. 
\end{proof}


\noindent
Equipped with Lemma~\ref{lem:differential} and Corollary~\ref{cor:CS}, we are ready to prove the monotonicity of MA (cf.~Algorithm~\ref{algo:MA}). 
For convenience,  let us denote the support of $w^k$ by $\calI_k\subseteq[n]$, 
i.e., 
\begin{equation}
\calI_k: =\{i\in[n]: w^k_i >0\}, \quad \forall\, k\ge 0.  
\end{equation}
%

\begin{theorem}[Monotonicity of MA]\label{thm:monotone}
Consider $\psi\in\bbV$ that  satisfies~\ref{assump:iso} and~\ref{assump:concave}, and let $\phi:=\rvT(\psi)\in\bbV$. 
In Algorithm~\ref{algo:MA}, assume that for some $k\ge 0$, 
$w^k\in\Delta_+^n$ and $\nabla_i f(w^k)>0$ for all $i\in\calI_k$.  Then for any $\lambda\in(0,1]$, we have $w^{k+1}\in\Delta_+^n$ and 
$f(w^{k+1})\ge f(w^{k})$. 

\end{theorem}

\begin{proof}
For convenience, define $M^k:= M(w^k)$ for $k\ge 0$. 
Since $w^k\in\Delta_+^n$ and $\nabla_i f(w^k)>0$ for all $i\in\calI_k$, we know that $\calI_{k+1} = \calI_k$ and  $w^{k+1}\in\Delta_+^n$. Define $$\textstyle\gamma_k:= \sum_{i\in\calI_k}w_i^k\,{\nabla_i f(w^k)^\lambda}>0,$$ so that $w^{k+1} = w^k\circ \nabla f(w^k)^\lambda/\gamma_k$. 
By setting $\alpha_i =  w^k_i\nabla_i f(w^k)/\gamma_k$, $\beta_i = \gamma_k w^k_i/\nabla_i f(w^k)$ and $V_i = A_i$ 
in Corollary~\ref{cor:CS}, we have
\begin{equation}
0\prec M^{k}(M^{k+1})^{-1}M^{k} \preceq \tilM^k, \quad \mbox{where}\;\; \tilM^k:= \gamma_k \textstyle\sum_{i\in\calI_k} (w^k_i/\nabla_i f(w^k)^\lambda) A_i,\label{eq:def_tilM}
\end{equation} 
which amounts to $0\prec (M^{k+1})^{-1} \preceq (M^{k})^{-1}\tilM^k(M^{k})^{-1}$. 
By the isotonicity and concavity of $\psi$ (cf.~\ref{assump:iso} and~\ref{assump:concave}), we have 
\begin{align}
\phi(M^{k+1}) = -\psi((M^{k+1})^{-1}) &\ge -\psi((M^{k})^{-1}\tilM^k(M^{k})^{-1}) \label{eq:monotone_0}\\
&\ge - \psi((M^{k})^{-1}) - \ipt{\nabla \psi((M^{k})^{-1})}{(M^{k})^{-1}\tilM^k(M^{k})^{-1} - (M^{k})^{-1}}\label{eq:monotone_0.5}\\
&= \phi(M^{k}) - (\ipt{\nabla \phi(M^{k})}{\tilM^k} -  \ipt{\nabla \phi(M^{k})}{M^{k}}),  \label{eq:phi_lb} 
\end{align}
where~\eqref{eq:phi_lb} follows 
from Lemma~\ref{lem:differential}. 
Now, by the definition of $\tilM^k$ in~\eqref{eq:def_tilM}, we have 
\begin{align}
\ipt{\nabla \phi(M^{k})}{\tilM^k} &= \gamma_k \textstyle\sum_{i\in\calI_k} (w^k_i/\nabla_i f(w^k)^\lambda) \ipt{\nabla \phi(M^{k})}{A_i}\nn\\
 &= \textstyle\big(\sum_{i\in\calI_k} w^k_i\,\nabla_i f(w^k)^\lambda\big) \big(\sum_{i\in\calI_k} w^k_i\nabla_i f(w^k)^{1-\lambda}\big)\label{eq:monotone_1}\\
 &\le \textstyle\big(\sum_{i\in\calI_k} w^k_i\,\nabla_i f(w^k)\big)^\lambda \big(\sum_{i\in\calI_k} w^k_i\nabla_i f(w^k)\big)^{1-\lambda}\label{eq:monotone_2}\\
 & = \textstyle\sum_{i\in\calI_k} w^k_i\,\ipt{\nabla \phi(M^{k})}{A_i}\label{eq:monotone_3}\\
 & = \ipt{\nabla \phi(M^k)}{M^k}, \label{eq:monotone_4}
\end{align}
where we use $\nabla_i f(w^k) = \ipt{\nabla \phi(M^{k})}{A_i}$ in~\eqref{eq:monotone_1} and~\eqref{eq:monotone_3} and the concavity of the functions $t\mapsto t^\lambda$ and $t\mapsto t^{1-\lambda}$ on $[0,+\infty)$ for $\lambda\in(0,1]$ in~\eqref{eq:monotone_2}. Combining~\eqref{eq:phi_lb} and~\eqref{eq:monotone_4}, we complete the proof. 
\end{proof}

\begin{remark}
Note that  by Remark~\ref{rmk:concave}, $\phi:= \rvT(\psi)$ need not be concave 
when $\psi$ satisfies~\ref{assump:iso} to~\ref{assump:concave}. Therefore, Theorem~\ref{thm:monotone} states that under certain conditions, Algorithm~\ref{algo:MA} is monotonic on~\eqref{eq:OED}  even if~\eqref{eq:OED} is a nonconvex problem. However, note that this does not imply that Algorithm~\ref{algo:MA} can {\em solve}~\eqref{eq:OED} when it is nonconvex. Indeed, (strict) concavity of $\phi$ on $\bbS_{++}^d$ is needed in~\cite[Theorem~3]{Yu_10} to show that  $\{f(w^k)\}_{k\ge 0}$ converges to $f^*$. 
\end{remark}

\begin{remark} \label{rmk:pos_grad}
In Theorem~\ref{thm:monotone}, the condition that $\nabla_i f(w^k)>0$ for all $i\in\calI_k$ is crucial to ensure that $w^{k+1}\in\Delta_+^n$ and the inequality in~\eqref{eq:def_tilM} 
holds. Note that this condition was also used in the proof of~\cite[Theorem~1]{Yu_10}, although it was not explicitly stated in that theorem. In addition, note that if 
\begin{equation}
\nabla \phi(M)\succ 0, \quad \, \forall\, M\succ 0,  \label{eq:PD_grad}
\end{equation}
 then we have $\nabla_i f(w)>0$ for all $w\in\Delta_+^n$ and $i\in[n]$. 
 Using 
 Lemma~\ref{lem:differential}, we know that $\nabla \phi(M)\succ 0$ for all $M\succ 0$ 
 if and only if $\nabla \psi(M)\succ 0$ for all $M\succ 0$, where $\psi:=\rvT(\phi)$.  
 Thus we easily see that $\phi_\rmD$,  $\phi_\rmA$ and $\phi_p$ 
 with $p>0$ all satisfy~\eqref{eq:PD_grad}. 
 That said, note that the  c-criterion, which is given by 
 \begin{equation}
 \phi_c(X) := -c^\top X^{-1} c, \quad  \forall\, X\succ 0,  \quad\mbox{where}\;\; c\ne 0, \label{eq:def_phi_c}
 \end{equation}
 may not be {strictly isotonic} on $\bbS_{++}^d$. Indeed, for this criterion, 
 the condition that $\nabla_i f(w^k)>0$ for all $i\in\calI_k$ may fail for some $w^k\in\Delta_+^n$ --- see Example~\ref{eg:3d} below for details. 
\end{remark}

\subsection{Strict monotonicity of MA}

Our simple and straightforward proof of the strict monotonicity of MA (cf.~Theorem~\ref{thm:monotone}) allows us to easily investigate the strict monotonicity of MA, which is important in proving the convergence of MA (cf.~\cite[Theorem~3]{Yu_10}).
As we can see, there are only three inequalities used in the proof of Theorem~\ref{thm:monotone}, and strict monotonicity  holds if at least one of these inequalities holds strictly. This leads to the following results. 

\begin{prop} \label{prop:lambda_(0,1)}
Consider the setting in Theorem~\ref{thm:monotone}. If 
$w^{k+1}\ne w^{k}$, or equivalently, 
\begin{equation}
\exists \, i,j\in \calI_k \quad \mbox{such that}\quad \nabla_i f(w^k) \ne \nabla_j f(w^k), \label{eq:nonuniform}
\end{equation}
then for any $\lambda\in(0,1)$, we have $f(w^{k+1})> f(w^{k})$. In addition, 
if $\calI_k=[n]$, then~\eqref{eq:nonuniform} holds  if and only if $w^k\not\in\calW^*$, where $\calW^*$ denotes the set of optimal solutions of~\eqref{eq:OED}, i.e., 
\begin{equation}
\calW^*:= \textstyle \argmax_{w\in\Delta_n } f(w). 
\end{equation}
\end{prop}

\begin{proof}
Note that the functions $t\mapsto t^\lambda$ and $t\mapsto t^{1-\lambda}$ are strictly concave on $[0,+\infty)$ for $\lambda\in(0,1)$. Thus under~\eqref{eq:nonuniform}, we see that the inequality~\eqref{eq:monotone_2} becomes strict. 
Next, consider 
that 
$\calI_k=[n]$. 
Since $\calW^*\subseteq \Delta_+^n$, on which $f$ is differentiable, by the first-order optimality condition of~\eqref{eq:OED}, we know that $\barw\in \calW^*$ if and only if 
\begin{equation}
\nabla_i f(\barw) = \textstyle\max_{i\in\bar \calI}\, \nabla_i f(\barw), \quad \forall\, i\in\bar\calI, 
\;\;\; \mbox{and}\;\;\;
\nabla_i f(\barw) \le  \textstyle\max_{i\in\bar\calI}\, \nabla_i f(\barw), \quad \forall\, i\in[n]\setminus\bar\calI,\label{eq:opt}
\end{equation}
where $\bar\calI$ denotes the support of $\barw$, i.e., $\bar\calI:= 
  \{i\in[n]: \barw_i >0\}$. 
Since $\calI_k=[n]$, by~\eqref{eq:opt}, we know that $w^k\in\calW^*$ if and only if $\nabla_i f(w^k)= \max_{i\in[n]}\, \nabla_i f(w^k)$ for all $i\in[n]$, which amounts to that~\eqref{eq:nonuniform} fails to hold.  
\end{proof}


\noindent
Proposition~\ref{prop:lambda_(0,1)} states that under the same setting of Theorem~\ref{thm:monotone}, as long as  $\lambda\in(0,1)$, we essentially obtain the strict monotonicity of MA ``for free'' (since $w^{k+1}\ne w^{k}$ is the minimal assumption for strict monotonicity to hold). In addition, we can easily obtain the following corollary.

\begin{corollary}
Consider $\psi\in\bbV$ that  satisfies~\ref{assump:iso},~\ref{assump:concave} and~\eqref{eq:PD_grad}, and let $\phi:=\rvT(\psi)$. 
In Algorithm~\ref{algo:MA}, choose any  $\lambda\in(0,1)$ and $w^0 \in \ri \Delta_n$. 
Then we have 
$f(w^{k+1})> f(w^{k})$ unless $w^k\in \calW^*$. 
\end{corollary}

\begin{proof}
From Remark~\ref{rmk:pos_grad}, we know that if $\psi$  satisfies~\eqref{eq:PD_grad}, so does   $\phi$. Consequently, 
if $w^k\in \ri\Delta_n$, then $\nabla f(w^k)>0$ and hence $w^{k+1}\in \ri\Delta_n$. Since $w^0 \in \ri \Delta_n$, we have $w^k\in \ri \Delta_n$ for all $k\ge 0$. Therefore, from Proposition~\ref{prop:lambda_(0,1)}, we know that 
if $w^k\not\in \calW^*$ and $\lambda\in(0,1)$, then 
$f(w^{k+1})> f(w^{k})$. 
\end{proof}

\noindent 
How about the case where $\lambda=1$? In this case,~\eqref{eq:monotone_2} holds with equality, and thus we have to consider sufficient conditions that lead to 
strict inequalities in~\eqref{eq:monotone_0} or~\eqref{eq:monotone_0.5}  (or both). 
To that end, we need to impose stronger assumptions on $\psi$ than those in~\ref{assump:iso} and~\ref{assump:concave}.  
Specifically, we require $\psi$ to be {\em strictly isotonic} on $\bbS_{++}^d$, i.e.,
\begin{equation}
0\prec A \preceq B, \;\; A\ne B \quad\Longrightarrow\quad  \psi(A) < \psi(B),  \label{eq:s_iso}
\end{equation}
and {\em strictly concave} on $\bbS_{++}^d$, i.e., 
\begin{equation}
A,B\succ 0, \;\; A\ne B \quad\Longrightarrow\quad  \psi(A) < \psi(B) + \ipt{\nabla \psi(B)}{A - B}.   \label{eq:s_iso}
\end{equation}

\begin{prop} \label{prop:lambda_1}
Consider the setting in Theorem~\ref{thm:monotone}, but with 
$\psi$ being strictly concave and strictly isotonic on $\bbS_{++}^d$. If 
$w^{k+1}\ne w^{k}$, 
then for any $\lambda\in(0,1]$, we have $f(w^{k+1})> f(w^{k})$. 
\end{prop}

\begin{proof}
By Proposition~\ref{prop:lambda_(0,1)}, it suffices to only consider $\lambda=1$. 
Suppose that $f(w^{k+1})= f(w^{k})$, which implies that both~\eqref{eq:monotone_0} and~\eqref{eq:monotone_0.5} hold with equality.  Since $\psi$ is strictly concave and strictly isotonic on $\bbS_{++}^d$, 
we have 
$M^{k}(M^{k+1})^{-1}M^{k} = \tilM^k$ and $\tilM^k = M^k= M^{k+1}$. By Corollary~\ref{cor:CS}, we  know that 
\begin{equation}
\textstyle\sqrt{\beta_i} A_i^{1/2} = \sqrt{\alpha_i}M^{k}(M^{k+1})^{-1} A_i^{1/2} = \sqrt{\alpha_i} A_i^{1/2}, \quad\forall\, i\in\calI_k,  
\end{equation}
where $\alpha_i =  w^k_i\nabla_i f(w^k)/\gamma_k$ and $\beta_i = \gamma_k w^k_i/\nabla_i f(w^k)$. Since $A_i\ne 0$, we have $\alpha_i = \beta_i$ for all $i\in\calI_k$, which implies that  $\nabla_i f(w^k)=\gamma_k$ for all $i\in\calI_k$, and hence $w^{k+1} = w^k$.  
\end{proof}

\noindent
Comparing Proposition~\ref{prop:lambda_1} with Proposition~\ref{prop:lambda_(0,1)}, we see that in terms of obtaining the strict monotonicity of MA, choosing $\lambda\in(0,1)$ is more advantageous than choosing $\lambda=1$, since the former 
requires weaker assumptions on $\psi$. In fact, as will be illustrated in Example~\ref{eg:2d} below, in some cases, choosing $\lambda\in(0,1)$ and $\lambda=1$ can lead to drastically different 
behaviors of MA. 


\vspace{1ex}
\noindent 
Finally, before ending this section, we provide a sufficient condition that ensures $\psi$ to be  strictly concave and strictly isotonic on $\bbS_{++}^d$. To that end, given a univariate function $g:(0,+\infty)\to\bbR$ and $X\in\bbS_{++}^d$ with spectral decomposition $X = \sum_{i=1}^d \lambda_i u_iu_i^\top$, define $g(X):= \sum_{i=1}^d g(\lambda_i) u_iu_i^\top$. We call $g$ {\em matrix monotone} on $\bbS_{++}^d$ if 
\begin{equation}
A\succeq B \succ 0 \quad\Longrightarrow \quad g(A) \succeq g(B). 
\end{equation}
It is well-known that both functions $t\mapsto\ln t$ and $t\mapsto t^p$ for $p\in[0,1]$ are {matrix monotone} on $\bbS_{++}^d$. For more details, see~\cite[Chapter 4]{Hiai_14}. 

\begin{prop}
Consider an injective  and strictly concave function $g:(0,+\infty)\to\bbR$ that is {matrix monotone} on $\bbS_{++}^d$. If $\psi(X) = \tr(g(X))$ for $X\in \bbS_{++}^d$, then $\psi$ is  strictly concave and strictly isotonic on $\bbS_{++}^d$. In particular, $\psi(X) = \ln\det(X) = \tr(\ln(X))$ and $\psi(X) = \tr(X^p)$ for $p\in(0,1)$ are strictly concave and strictly isotonic on $\bbS_{++}^d$.
\end{prop}

\begin{proof}
Note that $\tr(\cdot)$ is strictly isotonic on $\bbS^d$, namely, if $A\succeq B$ but $A\ne B$, then $\tr(A-B)  
>0$, and hence $\tr(A)> \tr(B)$.   
Since $g$ is injective and {matrix monotone} on $\bbS_{++}^d$, for any $A\succeq B \succ 0$ and $A\ne B$, we have $g(A)\succeq g(B)$ and $g(A)\ne g(B)$. As a result, we have $$\psi(A) = \tr(g(A))>\tr(g(B))=\psi(B).$$ 
In addition, since $\psi(X) = \sum_{i=1}^d g(\lambda_i(X))$, by~\cite[Theorem 4.5]{Bauschke_01}, the strict concavity of $\psi$ on $\bbS_{++}^d$ follows from the strict concavity of $g$ on $(0,+\infty)$. 
\end{proof}




\subsection{Illustrating Examples}

\begin{example}\label{eg:2d}
Let $n=d=2$, $A_i = e_i^\top e_i$ for $i\in[n]$, and the optimality criterion $\phi= \phi_\rmA$ (cf.~Section~\ref{sec:intro}).  
Here $e_i$ denotes the $i$-th standard coordinate vector, 
and $e$ denotes the vector with all entries equal to one. Consequently,~\eqref{eq:OED} becomes 
\begin{equation}
\textstyle f^*:= \sup\; \big\{f(w):=-(w_1^{-1} + w_2^{-1}) - \iota_{>0}(w_1) -\iota_{>0}(w_2)\big\} \quad \st\quad w_1+w_2 = 1,  
\label{eq:2d}
\end{equation}
where $\iota_{>0}$ denotes the indicator function of $(0,+\infty)$, namely, $\iota_{>0}(t) = 0$ if $t> 0$ and $+\infty$ if $t\le 0$.
Clearly,  
for~\eqref{eq:2d}, the feasible region $\Delta_n^+ = \ri\Delta_n$ and the unique optimal solution is $w^* = (1/2,1/2)^\top$ with $f^* = -4$.  Note that Algorithm~\ref{algo:MA} can be written as the following fixed-point iteration: 
\begin{equation}
\forall\,k\ge 0:\quad  w^{k+1}:= \rvF_\lambda(w^k),\quad 
\mbox{where}\quad  \rvF_\lambda(w) := \frac{w\circ \nabla f(w)^\lambda}{\ipt{w}{\nabla f(w)^\lambda}}, \quad \forall\,w\in\ri \Delta_n. 
\end{equation}
Observe that for~\eqref{eq:2d}, we have $\nabla f(w)= (w_1^{-2},w_2^{-2})^\top >0$ for $w\in\ri \Delta_n$, and hence 
$\rvF_\lambda$ has the following simple form:  
\begin{equation}
\rvF_\lambda(w) := \bigg(\frac{w_1^{1-2\lambda}}{w_1^{1-2\lambda} + w_2^{1-2\lambda}},\;\frac{w_2^{1-2\lambda}}{w_1^{1-2\lambda} + w_2^{1-2\lambda}}\bigg)\in\ri\Delta_n,\quad \forall\,w\in\ri\Delta_n.  \label{eq:def_F_lam}
\end{equation}
Let us make several interesting observations about $\rvF_\lambda$: 

\begin{enumerate}[label=(O\arabic*),leftmargin=6ex]
\item For any $w\in\ri \Delta_n$, we have $\rvF_1(w) = (w_2,w_1)$.  Hence if $w^0\ne w^*$ and $\lambda=1$, then Algorithm~\ref{algo:MA} will generate $\{w^k\}_{k\ge 0}$ that cycle between $(w^0_1,w^0_2)$ and $(w^0_2,w^0_1)$, and fail to converge. 
\item For any $w\in\ri \Delta_n$, we have $\rvF_{1/2}(w) = (1/2,1/2)$. Hence if $\lambda=1/2$, then for any $w^0\in \ri \Delta_n$, Algorithm~\ref{algo:MA} will reach $w^*$ in at most one step. 
\item \label{obs:1/2} For any $w\in\ri \Delta_n$ and $\varepsilon\in(0,1/2)$, we have 
\begin{equation}
\rvF_{\frac{1}{2}-\varepsilon}(w) := \bigg(\frac{w_1^{2\varepsilon}}{w_1^{2\varepsilon} + w_2^{2\varepsilon}},\;\frac{w_2^{2\varepsilon}}{w_1^{2\varepsilon} + w_2^{2\varepsilon}}\bigg) \quad\mbox{and}\quad \rvF_{\frac{1}{2}+\varepsilon}(w) := \bigg(\frac{w_2^{2\varepsilon}}{w_1^{2\varepsilon} + w_2^{2\varepsilon}},\;\frac{w_1^{2\varepsilon}}{w_1^{2\varepsilon} + w_2^{2\varepsilon}}\bigg). 
\end{equation}
Therefore, let $\{w^k\}_{k\ge 0}$ and $\{\tilw^k\}_{k\ge 0}$ be the iterates produced by Algorithm~\ref{algo:MA} with $\lambda = 1/2-\varepsilon$ and $\lambda = 1/2+\varepsilon$, respectively (and with the same starting point). Then $w^k = \tilw^k$ for even $k$ and $w^k = (\tilw^k_2,\tilw^k_1)$ for odd $k$. As a result, we have $f(w^k) = f(\tilw^k)$ for all $k\ge 0$. 
\end{enumerate}

\noindent 
Lastly, note that for~\eqref{eq:2d}, Algorithm~\ref{algo:MA} achieves global linear convergence in terms of the objective gap, and the linear rate is given by $\abst{1-2\lambda}$. 

\begin{prop}\label{prop:linear}
In~\eqref{eq:2d}, for any $w\in\ri \Delta_n$ and $\lambda\in(0,1)$, define $w^+:= F_{\lambda}(w)$. Then we have 
\begin{equation}
f^* - f(w^+)\le \abst{1-2\lambda} ( f^*-f(w) ). 
\end{equation}
\end{prop}

\begin{proof}
See Appendix~\ref{app:proof_linear}. 
\end{proof}

\end{example}

\begin{example} \label{eg:3d}
Let $n=d=3$, $A_i = e_i^\top e_i$ for $i\in[n]$ 
and $\phi = \phi_c$ 
for $c := (1,1,0)^\top$ (cf.~\eqref{eq:def_phi_c}). In this case,~\eqref{eq:OED} 
becomes 
\begin{equation}
\textstyle f^*:= \sup\; \big\{f(w):=-(w_1^{-1} + w_2^{-1})- \iota_{>0}(w_1) -\iota_{>0}(w_2)-\iota_{>0}(w_3)\big\} \quad \st\quad w_1+w_2+w_3 = 1. 
\label{eq:3d}
\end{equation}
Note that $f^* = -4$ but~\eqref{eq:3d} does not have any optimal solution. Also, the feasible region $\Delta_n^+ = \ri\Delta_n$.  In addition, note that for any $w\in \ri\Delta_n$, we have $\nabla f(w)= (w_1^{-2},w_2^{-2},0)^\top $, and hence for all $\lambda\in(0,1]$, the next iterate $w^+:= \rvF_\lambda(w) \not\in \ri\Delta_n$. In fact, we have $f(w^+)=-\infty$ and $f$ is not differentiable 
at $w^+$.   This implies that in Algorithm~\ref{algo:MA}, for any starting point $w^0\in \ri\Delta_n$, we have $w^1\not\in \ri\Delta_n$ and $f$ is not differentiable 
at $w^1$.  
This prevents Algorithm~\ref{algo:MA} from proceeding further.
\end{example}

\noindent
 Examples~\ref{eg:2d} and~\ref{eg:3d} may look similar on the surface. 
 However, note that Algorithm~\ref{algo:MA} exhibits  
vastly different behaviors on these two examples. 
This can be partly attributed to the fact that $\nabla f(w)>0$ for all $w\in \ri\Delta_n$ in  Example~\ref{eg:2d}, 
but it is not the case in Example~\ref{eg:3d}. The positivity of $\nabla f$ on $\ri\Delta_n$ ensures that if $w^0\in\ri\Delta_n$, then $w^k\in\ri\Delta_n$ for all $k\ge 0$, 
which is precisely required in the monotone convergence theory of MA in~\cite[Theorem~2]{Yu_10}. 
To certain extent, Example~\ref{eg:3d} reveals some limitations of MA on OED problems where the condition $\nabla f(w)>0$ for all $w\in \ri\Delta_n$ fails. This issue will be discussed in more details in the next section. 


\section{Discussions and Open Problems} \label{sec:discuss}

Note that the formulation in~\eqref{eq:OED} restricts the feasible moment matrix $M(w)$ to be positive definite, which makes sense if we wish to estimate the full parameter $\theta$ (cf.~Section~\ref{sec:intro}). 
However, as illustrated in Pukelsheim~\cite{Pukel_06}, if one is interested in estimating a linear parameter subsystem $K^\top \theta$, where $K\in \bbR^{d\times s}$ has full column rank $s$, then we can relax the requirement $M(w)\succ 0$ to $M(w)\in\calF(K)$, where $\calF(K)$ is called the {\em feasibility cone} (induced by $K$) and given by
\begin{equation}
\calF(K):= \{M\in \bbS_+^d: \calR(K)\subseteq \calR(M) \}. 
\end{equation}
Here $\calR(B)$ denotes the range (or column space) of a matrix $B$. Note that $\bbS_{++}^d\subseteq \calF(K)\subseteq \bbS_+^d$, and $\calF(K)$ is a convex cone in the following sense: 
\begin{equation}
\beta M + \beta' M' \in \calF(K), \quad \forall\, M, M'\in \calF(K), \;\; \beta,\beta' > 0. 
\end{equation}
In fact, for any $M\in \calF(K)$, we can define the following {\em information matrix}:
\begin{equation}
C_K(M): = (K^\top M^\dagger K)^{-1}, 
\end{equation}
where $ M^\dagger$ denotes the pseudo-inverse of $M$. Let us note the following two extreme cases: 
\begin{enumerate}[label = (C\arabic*),leftmargin = 6ex]
\item \label{case:K_Id} If $K =I_d$, then $\calF(K)  = \bbS_{++}^d$, and $C_K(M) = M$ for any $M\in \calF(K) $. 
\item If $K= c\in\bbR^d\setminus\{0\}$, then $\calF(K)  = \{M\in \bbS_+^d: c\in  \calR(M) \}$, and $C_K(M) = (c^\top M^\dagger c)^{-1}$ for any $M\in \calF(K) $. 
\end{enumerate}
Based on $C_K(M)$ and an optimality criterion $\phi:\bbS^s_{++}\to \bbR$, we can define $\Gamma:\bbS^n\to \underline \bbR$ 
such that 

\begin{equation}
   \Gamma(M)  := \bigg\{ 
   \begin{aligned}
   \;\phi(C_K(M)), &\quad   M\in \calF(K)  \\
   -\infty, &\quad    \mbox{otherwise}
\end{aligned}\;\;, 
\end{equation}
and we solve a ``generalized'' formulation of~\eqref{eq:OED} in the following:
 \begin{equation}
\textstyle \sup_{w\in\Delta_n }\; \Gamma(M(w)).  \tag{${\rm OED_1}$} \label{eq:OED1} 
\end{equation}
Let us make a few quick remarks about~\eqref{eq:OED1}. First, from~\ref{case:K_Id}, we know that~\eqref{eq:OED1} reduces to~\eqref{eq:OED} when $K =I_d$. Second, as mentioned in Lu and Pong~\cite[Remark~3.1(b)]{Lu_13}, under certain conditions on $\phi$ (which is satisfied by $\phi_\rmD$ and $\phi_p$ for $p>0$),~\eqref{eq:OED1} has an optimal solution. Third, according to Pukelsheim~\cite[Sections~7.6~to~7.8]{Pukel_06}, if $\phi$ is isotonic, concave and sub-differentiable on $\bbS_{++}^d$, then $\Gamma$ is concave and sub-differentiable on $\calF(K)$. Note that we call $\Gamma$ 
sub-differentiable on $\calF(K)$ if for all $M\in\calF(K)$, we have $\partial \Gamma(M)\ne \emptyset$, where    
\begin{equation}
\partial \Gamma(M):=\{G\in\bbS^d: \Gamma(M') \le \Gamma(M) + \ipt{G}{M'-M}, \quad \forall\, M'\in\calF(K)\} . 
\end{equation}

Let us turn our focus back to Example~\ref{eg:3d}, where $n=d=3$ and $A_i = e_i^\top e_i$ for $i\in[n]$. 
In~\eqref{eq:OED1}, if we let 
 $K= c=(1,1,0)$, then we have $\calF(K) = \{w\in\bbR^3: w_1,w_2>0, w_3\ge 0\}$ and $C_K(M) = (w_1^{-1} + w_2^{-1})^{-1}$. Furthermore, if we let $\phi(t) = -1/t$ for $t>0$, then~\eqref{eq:OED1} becomes  
\begin{equation}
\textstyle F^*=\sup\; \big\{F(w):=-(w_1^{-1} + w_2^{-1})- \iota_{>0}(w_1) -\iota_{>0}(w_2)-\iota_{\ge 0}(w_3)\big\} \quad \st\quad w_1+w_2+w_3 = 1, 
\label{eq:3d2}
\end{equation}
where $\iota_{\ge 0}$ denotes the indicator function of $[0,+\infty)$. 
Compared to~\eqref{eq:3d}, the only difference in~\eqref{eq:3d2} is that in the objective function, we have $\iota_{\ge 0}(w_3)$ instead of $\iota_{> 0}(w_3)$. Note that this difference yields at least two important consequences. First, note that $-F$ is proper, convex and {\em closed}, 
and hence~\eqref{eq:3d2} has an optimal solution. In fact, in this case, the optimal solution is unique and given by $w^* = (1/2,1/2,0)$. Second, for any $w\in \calF(K)\setminus \bbS_{++}^d$, 
we have 
$\partial F(w)= \{(w_1^{-2},w_2^{-2},t): t\ge 0\}$. 
Now, if we apply Algorithm~\ref{algo:MA} to solve~\eqref{eq:3d2} with any starting point $w^0\in \ri\Delta_n$,  as observed in Example~\ref{eg:3d}, we will have $w^1_3=0$, and $\nabla F(w^1)$ is undefined. However, if we modify Algorithm~\ref{algo:MA} such that  $\nabla f(w^{k})$ is replaced by any $g^k\in \partial F(w^k)$ in Step~\ref{step:multiply}, then the modified algorithm can still proceed from 
$w^1 = (w^1_1,w^1_2,0)$, and generate the iterates $\{(w^k_1,w^k_2,0)\}_{k\ge 1}$ in the following way:
\begin{equation}
\forall\,k\ge 1:\quad  (w_1^{k+1},w_2^{k+1}):= \rvF_\lambda((w_1^{k},w_2^{k})), 
\end{equation}
where $\rvF_\lambda$ is defined in~\eqref{eq:def_F_lam}. As a result, if $\lambda\in(0,1)$, then we know that $\{(w^k_1,w^k_2,0)\}_{k\ge 1}$ converges to  $w^*$ linearly in terms of the objective value, i.e., 
\begin{equation}
F^* - F(w^{k+1})\le \abst{1-2\lambda} ( F^*-F(w^k) ), \quad  \forall\, k\ge 1. 
\end{equation}

In fact, we may interpret the modification above in a different, yet more intuitive way. 
Once $w^1_3= 0$, by the multiplicative nature of Algorithm~\ref{algo:MA}, we will have $w_3^k= 0$ for all $k\ge 1$. As such,  we can reduce 
the problem in~\eqref{eq:3d2} to that in~\eqref{eq:2d} by dropping the third coordinate, and apply Algorithm~\ref{algo:MA} to the reduced problem~\eqref{eq:2d}. 
For the particular case of~\eqref{eq:3d2}, 
from the discussions 
above, we know that this approach works, namely, starting from any $w^0\in \ri\Delta_n$, 
the generated sequence $\{w^k\}_{k\ge 0}$ convergences to $w^*$ linearly (in terms of the objective value). 
However, it remains an open question whether such a ``coordinate dropping'' approach works in general. Specifically, starting from $w^0\in \ri\Delta_n$, if for some $k\ge 0$ and $i\in[n]$, we have $w^k_i>0$ but $w^{k+1}_i=0$, then the following three questions naturally arise:
\begin{itemize}
\item (Feasibility) Is it true that $M(w^{k+1})\in \calF(K)$?
\item (Monotonicity) Do we have $f(w^{k+1})\ge  f(w^k)$? (Note that the monotonicity result in Theorem~\ref{thm:monotone} requires $w^k$ and $w^{k+1}$ to have the same support.) 
\item (Convergence to optimum) Is there 
an optimal solution $w^*$ of~\eqref{eq:OED1} such that $w^*_i=0$?
\end{itemize}

\noindent
If the answers to all of the three questions above are affirmative, then we may have a more general monotone convergence theory than~\cite[Theorem~2]{Yu_10}, which indeed requires $w^k\in \ri\Delta_n$ for all $k\ge 0$. We believe that all of these questions are worth further investigations, and will lead to deeper understanding of the behaviors of MA when applied to solving~\eqref{eq:OED1}.



\appendix

\section{Proof of Lemma~\ref{lem:differential}} \label{app:diff}

For any $H\in\bbS^d$, define $g(t): = \phi(M+tH)$ with $\dom g:= \{t\in\bbR:M+tH\succ 0\}$. 
By the definition of $\psi$ in~\eqref{eq:def_psi}, we know that $\phi(M) = -\psi(M^{-1})$ for $M\succ 0$, and therefore, 
\begin{align*}
g(t) = -\psi\big((M+tH)^{-1}\big) &= -\psi\big(M^{-1/2}(I+t\tilH)^{-1}M^{-1/2}\big), \quad \mbox{where}\;\; \tilH := M^{-1/2} H M^{-1/2}.
\end{align*}
Now, write the spectral decomposition of $\tilH$ as $\tilH = \sum_{i=1}^d \lambda_i u_iu_i^\top$, and we have 
\begin{align*}
g(t) = -\psi\big(M(t)\big), \quad \mbox{where}\;\; M(t):= (M+tH)^{-1} = \textstyle\sum_{i=1}^m (1+t\lambda_i)^{-1} M^{-1/2}u_iu_i^\top M^{-1/2}.
\end{align*}
Since $g'(t) = -\ipt{\nabla \psi\big(M(t)\big)}{ M'(t)} $ and $M'(t)= M^{-1/2} \textstyle(\sum_{i=1}^m -(1+t\lambda_i)^{-2} \lambda_i u_iu_i^\top) M^{-1/2}$, we have
\begin{align*}
 g'(0) = -\ipt{\nabla \psi\big(M(0)\big)}{ M'(0)} = \ipt{\nabla \psi\big(M^{-1}\big)}{M^{-1/2} \tilH M^{-1/2} } =  \ipt{\nabla \psi\big(M^{-1}\big)}{M^{-1} H M^{-1} }. 
\end{align*} 
Since we also have $g'(0) = \ipt{\nabla \phi(M)}{H}$,  the proof is complete. 

\section{Proof of Lemma~\ref{lem:MCS}} \label{app:MCS}

Let $C :=  (\sum_{i=1}^n B_iA_i^\top) (\sum_{i=1}^n A_iA_i^\top)^{-1}$, and we have 
\begin{align}
 0&\preceq \textstyle\sum_{i=1}^n \textstyle (B_i - CA_i) (B_i - CA_i)^\top \label{eq:prec}\\
& = \textstyle\sum_{i=1}^n B_iB_i^\top + C(\sum_{i=1}^n A_iA_i^\top) C^\top - C (\sum_{i=1}^n A_iB_i^\top) - (\sum_{i=1}^n B_iA_i^\top) C^\top\\
&= \textstyle\sum_{i=1}^n B_iB_i^\top  - (\sum_{i=1}^n B_iA_i^\top) (\sum_{i=1}^n A_iA_i^\top)^{-1} (\sum_{i=1}^n A_iB_i^\top). 
\end{align}
In addition, note that~\eqref{eq:prec} holds with equality if and only if $B_i = CA_i$ for all $i\in[n]$. 

\section{Proof of Proposition~\ref{prop:linear} } \label{app:proof_linear}

Let us first consider $\lambda\in(0,1/2]$ and define $\gamma := 1-2\lambda\in[0,1)$. Then
\begin{align*}
f^* - f(w^+) &= 2 + (w_1/w_2)^\gamma + (w_2/w_1)^\gamma - 4\\
&\le 2 + 1 + \gamma (w_1/w_2 -1) + 1 + \gamma (w_2/w_1 -1) - 4 \numberthis \label{eq:gamma_concave}\\
&= \gamma (w_1^2+w_2^2 - 2w_1w_2)/(w_1w_2)\\
&= \gamma ((w_1+w_2)^2/(w_1w_2) - 4)\\
&= \gamma (w_1^{-1} + w_2^{-1}- 4)\numberthis \label{eq:w_1_w_2} \\
&= (1-2\lambda) (f^* - f(w)), 
\end{align*}
where we use the concavity of $t\mapsto t^\gamma$ for $\gamma \in[0,1)$ in~\eqref{eq:gamma_concave}  and that $w_1+w_2  = 1$ in~\eqref{eq:w_1_w_2}. Next, note that if $\lambda\in(1/2,1)$, we know from~\ref{obs:1/2} above  that $f(w^+) = f(\tilw^+)$, where $\tilw^+ := F_{\lambda'}(w)$ and $\lambda':=1-\lambda\in(0,1/2)$. As a result, we have
\begin{align*}
f^*-f(w^+)  = f^*-f(\tilw^+)  \le (1-2\lambda')  ( f^*-f(w)) = (2\lambda-1)(f^*-f(w) ). \tag*{$\square$}
\end{align*}


\bibliographystyle{abbrv}
\bibliography{math_opt,mach_learn,stat_ref}

\end{document}